\documentclass[EJP]{ejpecp}

\usepackage{mathrsfs}
\usepackage[latin1]{inputenc}
\usepackage[T1]{fontenc}
\usepackage{wasysym}

\SHORTTITLE{Subcritical contact process seen from the edge}

\TITLE{Subcritical contact process seen from the edge: Convergence to quasi-equilibrium}

\AUTHORS{%
Enrique Andjel\footnote{Université d'Aix-Marseille, visting IMPA supported by CNPq}\ \ {}
François Ezanno\footnote{Lycée Roland Garros}\ \ {}
Pablo Groisman\footnotemark[3]\ \ {}
Leonardo T. Rolla\footnote{Universidad de Buenos Aires}}


\AMSSUBJ{60K35; 82C22}

\SUBMITTED{October 22, 2014}
\ACCEPTED{March 20, 2015}

\VOLUME{20}
\YEAR{2015}
\PAPERNUM{32}
\DOI{v20-3881}

\ABSTRACT{The subcritical contact process seen from the rightmost infected site has no invariant measures. We prove that nevertheless it converges in distribution to a quasi-stationary measure supported on finite configurations.}

\newtheorem*{theorem*}{Theorem}

\newcommand{\E}{\mathbb{E}}
\newcommand{\N}{\mathbb{N}}
\newcommand{\Pb}{\mathbb{P}}
\newcommand{\R}{\mathbb{R}}
\newcommand{\Z}{\mathbb{Z}}
\renewcommand{\L}{\mathscr{L}}
\newcommand{\LL}{\mathbb{L}}
\newcommand{\G}{\hat G}
\newcommand{\A}{\mathcal{A}}

\newcommand{\F}{\mathcal{F}}
\newcommand{\dd}{{\mathrm d}}
\renewcommand{\d}{\dd}
\renewcommand{\leq}{\leqslant}
\renewcommand{\le}{\leqslant}
\renewcommand{\geq}{\geqslant}
\renewcommand{\ge}{\geqslant}
\newcommand{\I}{\mathds{1}}

\newcommand{\sigmainfty}{\Sigma^{^{\scriptscriptstyle*}}}
\newcommand{\sigmainftybar}{\Sigma^{^{\scriptscriptstyle*}}_{\scriptscriptstyle0}}
\newcommand{\sigmabar}{\Sigma_{\scriptscriptstyle0}}

\begin{document}

\section{Introduction}
\label{sec:intro}

The contact process is a stochastic model for the spread of an infection among the members of a population.
Individuals are identified with points of a lattice ($\Z$ in our case) and the process evolves according to the following rules.
An infected individual will infect each of its neighbors at rate $\lambda>0$, and recover at rate $1$.
This evolution defines an interacting particle system whose state at time $t$ is a subset $\eta_t \subseteq \Z$, or equivalently an element $\eta_t \in \{0,1\}^\Z$.
We interpret that individual $x\in\Z$ is infected at time $t$ if $\eta_t(x)=1$, and is otherwise healthy.

The contact process is one of the simplest particle systems that exhibits a phase transition.
There exists a critical value $0<\lambda_c<\infty$ such that the probability that a single individual infects infinitely many others is zero when $\lambda < \lambda_c$ and is positive when $\lambda > \lambda_c$.
See~\cite{liggett-85} for the precise definition of the model.

For $A\subseteq \Z$, let $(\eta_t^A)_{t\ge 0}$ denote the process starting from $\eta_0=A$.
When $A$ is random and has distribution~$\mu$, we denote the process by $\eta_t^\mu$.
We also write $\eta_t^x$ when $A={\{x\}}$.
Let
\[
\Sigma
=
\{ A \subseteq \Z : A \cap \N \text{ is finite} \}
\quad
\text{ and }
\quad
\sigmainfty
=
\{ A \in \Sigma : A\cap -\N \text{ is infinite} \}
.
\]
Notice that both $\Sigma$ and $\sigmainfty$ are invariant for the contact process dynamics.
For $A\in\Sigma$, the \emph{contact process seen from the rightmost point} is the Markov process on~$\Sigma$ defined by
\[
\zeta_t^A = \eta_t^A - \max \eta_t^A
\]
if
$\eta_t^A \ne \emptyset$
and
$\zeta_t^A = \emptyset$
otherwise.
In fact, defining
\[
\sigmabar = \{A \in \Sigma : \max A=0 \text{ or } A = \emptyset \}
\quad
\text{and}
\quad
\sigmainftybar =\{A \in \sigmabar : A \text{ is infinite} \}
,
\]
the state-space of the process~$(\zeta_t)_t$ is $\sigmabar$, and the subset $\sigmainftybar$ is invariant.

Durrett~\cite{durrett-84} proved the existence of an invariant measure for $\zeta_t$ when $\lambda \geqslant \lambda_c$ on~$\sigmainftybar$.
In the supercritical phase, Galves and Presutti~\cite{galves-presutti-87} proved that the invariant measure~$\mu$ is unique for each $\lambda$, and that $\zeta_t^A$ converges in distribution to~$\mu$ for any $A\in\sigmainftybar$.
Kuczek~\cite{kuczek-89} provided an alternative proof and showed an invariance principle for the position of the rightmost infected site.
Uniqueness of~$\mu$ and convergence in distribution was extended to the critical case by Cox, Durrett, and Schinazi~\cite{cox-durret-schinazi-91}.
While some of these results were stated for the contact process and some others for planar oriented percolation, the arguments in~\cite{galves-presutti-87,kuczek-89,cox-durret-schinazi-91} can be translated effortlessly from one model to the other, which is not always the case.

The behavior in the subcritical phase is quite different.
Schonmann~\cite{schonmann-87} showed that in this phase, planar oriented percolation seen from its rightmost point does not have any invariant measure on~$\sigmainftybar$.
This result was extended to the contact process by Andjel, Schinazi, and Schonmann~\cite{andjel-schinazi-schonmann-90}.

In this paper we show that, despite non-existence of stationary measures, the subcritical contact process seen from the rightmost point does converge in distribution.
The limiting measure is quasi-stationary and is supported on configurations that contain finitely many infected sites.

This extends an analogous result for subcritical planar oriented percolation~\cite{andjel-14}.
The proof in~\cite{andjel-14} used quite heavily that in the discrete setting the speed of the propagation of the infection is bounded by $1$ almost surely.
Since this does not hold for the contact process, there is no simple adaptation of that proof to this case.
The difficulty is mostly due to the fact that unlikely events may have considerable influence when we observe an event of small probability.

Hereafter we assume that $0 < \lambda < \lambda_c$ is fixed.

For $\zeta_0\sim\mu$, define $\tau^\mu = \inf\{t\ge 0 : \zeta_t^\mu =\emptyset \}$.
Define $\tau^A$ and $\tau^x$ analogously.
We say that~$\mu$ is a \emph{quasi-stationary distribution} on $\sigmabar$ if for every $t>0$ the law of $\zeta_t^\mu$ satisfies
\[
 \L (\zeta_t^\mu \, |\, \tau^\mu >t) = \mu.
\]
By the Markov property, the above identity implies that $\L (\tau^\mu \, |\, \tau^\mu >t) = \L (\tau^\mu + t)$, so if~$\mu$ is quasi-stationary, $\tau^\mu$ is exponentially distributed with some parameter $\alpha=\frac{1}{\E[\tau^\mu]} \ge 0$.
We say that ${\nu}$ is \emph{minimal} if $\E[\tau^{\nu}]$ is minimal among all quasi-stationary distributions.
Notice that stationary is a particular case of quasi-stationary with $\alpha=0$.
See~\cite{vandoorn-pollet-13,collet-martinez-sanmartin-13,meleard-villemonais-12} for an introduction on this topic.

\begin{proposition}
\label{prop:yaglom}
The subcritical contact process seen from the rightmost point $(\zeta_t)_{t\geqslant 0}$ has a unique minimal quasi-stationary distribution~${\nu}$.
This measure~$\nu$ is supported on finite configurations.
Moreover it satisfies the Yaglom limit
\[
\L ( \zeta_t^A \,|\, \tau^A>t ) \to {\nu} \text{ as } t\to\infty
,
\]
for any finite configuration $A\subseteq \Z$.
\end{proposition}
An analogous result was obtained by Ferrari, Kesten, and Martínez~\cite{ferrari-kesten-martinez-96} for a class of probabilistic automata that includes planar oriented percolation.
The main step of their proof is to show that the transition matrix is $R$-positive with left eigenvector~${\nu}$ summable.
In our proof we show that the contact process observed at discrete times falls in that class, and then apply standard theory of $\alpha$-positive continuous-time Markov chains to obtain the Yaglom limit.
In Section~\ref{sec:yaglom} we state and prove a more general version of the above proposition, valid on~$\Z^d$.

We finally state our main result.

\begin{theorem}
\label{thm:2}
For every infinite initial configuration~$A \subseteq -\N$, the subcritical contact process seen from the rightmost point $\zeta_t^A$ converges in distribution to ${\nu}$ as $t\to\infty$.
\end{theorem}

Theorem~\ref{thm:2} is proved in Section~\ref{sec:main.theorem} using Proposition~\ref{prop:yaglom}.
A natural attempt to get Theorem~\ref{thm:2} would be to consider the rightmost site $x\in-\N$ whose infection survives up to time~$t$, and simply apply Proposition~\ref{prop:yaglom} to the set~$\zeta^x_t$ of sites infected by~$x$ at time~$t$.
However, the choice of~$x$ as the \emph{first} surviving site brings more information than simply ``$\tau^x > t$''.
We define a sort of renewal space-time point in order to handle this extra information, and finally show that such point exists with high probability.

\medskip

Some of the main arguments in this paper come from the second author's thesis~\cite{ezanno12}.

\section{The set infected by an infinite configuration}
\label{sec:main.theorem}

In this section we prove Theorem~\ref{thm:2}.	
Subsection~\ref{sec:graphical.construction} describes the graphical construction of the contact process, and 
in Subsection~\ref{sec:fkgbk} we recall the FKG and BK inequalities for this construction.

In Subsection~\ref{sec:breakpoints} we introduce the definitions of a \emph{good} space-time point, and a \emph{break point}, for fixed time~$t$.
The presence of a break point neutralizes the inconvenient information mentioned at the end of Section~\ref{sec:intro}, provided that all points nearby are good.
Choosing some constants correctly, it turns out that most points are good, even when considering rare regions such as those where an infection happens to survive until time~$t$.
We conclude this subsection with the proof of Theorem~\ref{thm:2}, assuming existence of break points.

In Subsection~\ref{sec:findbreakpoints} we prove that a break point can be found with high probability as $t\to\infty$.
To that end we use again geometric properties of good points and the exponential decay of subcritical contact processes.

\subsection{Graphical construction}
\label{sec:graphical.construction}

Define $\LL= \Z + \{ \pm 1/3 \}$ and let $U$ be a Poisson point process in $\R^2$ with intensity given by
\(
\big( \sum_{y\in\Z} \delta_y + \sum_{y\in\LL} \lambda\delta_y \big) \times \dd t
.
\)
Notice that $U \subseteq (\Z\cup \LL )\times \R$.
Let $(\Omega, \F, \Pb)$ be the underlying probability space.
For $x \in \Z$ we write $U^{x,x\pm1}=U\cap ( \{x\pm1/3\} \times \R )$ and $U^x=U\cap ( \{x\} \times \R )$.

Given two space-time points $(y,s)$ and $(x,t)$, we define a \emph{path from $(y,s)$ to $(x,t)$} as a finite sequence $(x_0,t_0), \dots, (x_k,t_k)$
with $x_0=x$, $x_k=y$, $s=t_0 \le t_1 \le \dots \le t_k = t$ with the following property.
For each $i=1,\dots,k$, the $i$-th segment $[(x_{i-1},t_{i-1}),(x_i,t_i)]$ is either \emph{vertical}, that is, $x_{i}=x_{i-1}$, or \emph{horizontal}, that is, $|x_{i}-x_{i-1}|=1$ and $t_{i}=t_{i-1}$.
Horizontal segments are also referred to as \emph{jumps}.
If all horizontal segments satisfy $t_i = t_{i-1} \in U^{x_{i-1},x_i}$ then such path is also called a \emph{$\lambda$-path}.
If, in addition, all vertical segments satisfy $(t_{i-1},t_{i}] \cap U^{x_i}=\emptyset$ we call it an \emph{open path} from $(y,s)$ to $(x,t)$.

The existence of an open path from $(y,s)$ to $(x,t)$ is denoted by
$ (y,s) \rightsquigarrow (x,t)$.
Also for two sets of the plane $C$, $D$ we use
$\{C \rightsquigarrow D\} = \{ (y,s) \rightsquigarrow (x,t) \text{ for some } (y,s)\in C, (x,t)\in D\}$.
We denote by $L_t$ the set $\Z \times \{t\} \subseteq \R^2$.

For $A \in \sigmabar$, we define $\eta_{s,t}^A \in \Sigma$ and $\zeta_{s,t}^A \in \sigmabar$ by
\begin{equation}
\label{def.contact.process}
\eta_{s,t}^A = \{ x : (A\times\{s\}) \rightsquigarrow (x,t)\}
,
\quad
\quad
\zeta_{s,t}^A = \eta_{s,t}^A - \max \eta_{s,t}^A
\end{equation}
if $\eta^A_{s,t}\ne \emptyset$ and $\zeta_{s,t}^A=\emptyset$ otherwise. When $s=0$ we omit it. We use $(\eta_t)$ and $(\zeta_t)$ for the processes defined by \eqref{def.contact.process}, that is $(\eta_t)$ is a contact process with parameter $\lambda$ and $(\zeta_t)$ is this process as seen from the rightmost infected site. Both of them are Markov. Note that if $A$ is finite, the same holds for $\eta_t^A$ and $\zeta_t^A$ for every $t\ge 0$ with total probability. Also note that $\emptyset$ is absorbing for both processes. When $A$ is a singleton~$\{y\}$ we write $\eta^y_t$ and $\zeta^y_t$.

\subsection{FKG and BK inequalities}
\label{sec:fkgbk}

We use $\omega$ for a \emph{configuration} of points in $\R^2$ and $\omega_\delta$, $\omega_\lambda$ for its restrictions to $\Z\times \R$ and $\LL\times \R$ respectively.
We write $\omega\ge \omega'$ if $\omega'_\lambda \subseteq \omega_\lambda$ and $\omega_\delta \subseteq \omega'_\delta$.

We slightly abuse the notation and identify a set of configurations~$Q$ with $(U^{-1} Q) \subseteq \Omega$.

A minor topological technicality needs to be mentioned.
Consider the space of locally finite configurations with the Skorohod topology: two configurations are close if they have the same number of points in a large space-time box and the position of the points are approximately the same.
In the sequel we assume that all events considered have zero-probability boundaries under this topology.
The important fact is that events of the form $\{E \rightsquigarrow F \}$ are measurable and satisfy this condition, as long as $E$ and $F$ are bounded Borel subsets of $\Z \times \R$.
See~\cite[Sect. 2.1]{bezuidenhout-grimmett-91} for a proof and precise definitions.

\begin{definition}
A set of configurations $Q$ is \emph{increasing} if $\omega \ge \omega' \in Q$ implies $\omega \in Q$.
\end{definition}

\begin{theorem*}
[FKG Inequality]
If $Q_1$ and $Q_2$ are increasing, then
\(
 \Pb(Q_1 \cap Q_2) \ge \Pb(Q_1)\Pb(Q_2).
\)
\end{theorem*}

\begin{definition}
Let $D$ denote a Borel subset of $\R^2$.
For a given configuration~$\omega$, we say that~\emph{$Q$ occurs on $D$} if $\omega' \in Q$ for any configuration~$\omega'$ such that $\omega' \cap D = \omega \cap D$. 
\end{definition}

\begin{definition}
We way that \emph{$Q_1$ and $Q_2$ occur disjointly} if there exist disjoint sets $D_1$ and $D_2$ such that $Q_1$ occurs on $D_1$ and $Q_2$ occurs on $D_2$.
This event is denoted by $Q_1 \Square\, Q_2$.
\end{definition}

\begin{theorem*}
[BK Inequality]
If $Q_1$ and $Q_2$ are increasing, and depend only on the configuration~$\omega$ within a bounded domain, then
\(
\Pb(Q_1 \Square\,Q_2) \le \Pb(Q_1)\Pb(Q_2).
\)
\end{theorem*}

For the proofs, see for instance~\cite[Sect.~2.2]{bezuidenhout-grimmett-91}.

\subsection{Good points and break points}
\label{sec:breakpoints}

The definitions below are parametrized by $t>0$ and $\beta>0$, but we omit it in the notation.
We write $\beta t$ as a short for $\lceil \beta t \rceil$.
For simplicity we assume through this whole section that the initial configuration $A \in \sigmainftybar$ is fixed.

\begin{definition}
[Good point]
We say that $(z,s)$ is \emph{good}, an event denoted by $G^s_z$, if every $\lambda$-path starting at $(z,s)$ makes less than $\beta t$ jumps during $[s,s+t]$.
We also denote $\hat G^s_z:=G^s_z \cap G^s_{z+2\beta t}.$
The time~$s$ is omitted when $s=0$.
We will say that $(z,s)$ is \emph{$(\beta,t)$-good} when we need to make $\beta$ and $t$ explicit.
\end{definition}

The above definition is helpful in this continuous-time context, as a way to recover independence of connectivity at distant regions.
As an example, observe that even though the events $\{\mathbf{0} \rightsquigarrow L_t\}$ and $\{(2\beta t,0) \rightsquigarrow L_t\}$ are not independent, they are conditionally independent given $\hat G_0$.
Hereafter we write $\mathbf{0}$ for the space-time point $(0,0)$.
Moreover, each point is typically good, so that conditioning on a large set of points to be good has negligible impact on the underlying distribution.

\begin{definition}
[Break point]
We say that the space-time point $(y,s)$ is a \emph{break point} if, for every $w \in (y,y+2\beta t] \cap \Z$, $L_0 \not \rightsquigarrow (w,s)$.
\end{definition}

Let
\( X := \max\{x\in A:(x,0) \rightsquigarrow L_t\} \)
denote the first site whose infection survives up to time~$t$, and let $\Gamma:[0,t]\to\Z$ given by
\[
\Gamma(s) := \max\{y:(X,0)\rightsquigarrow (y,s)\rightsquigarrow L_t\}
\]
denote the ``rightmost path'' from $(X,0)$ to $L_t$.
Take
\[ R := \inf\{s \in [0,t]:(s,\Gamma(s)) \text{ is a break point}\} \]
as the time of the first break point in~$\Gamma$, and let $Y:=\Gamma(R)$, see Figure~\ref{figure:definitions}.
Finally take
\[ \mathcal{A} := \{x \le 0 : A \times \{0\} \rightsquigarrow (Y+x,R)\} \]
as the set of sites infected at time~$R$, lying to the left of~$\Gamma$, seen from~$\Gamma$.
\begin{figure}[ht!]
\centering\includegraphics[height=60mm]{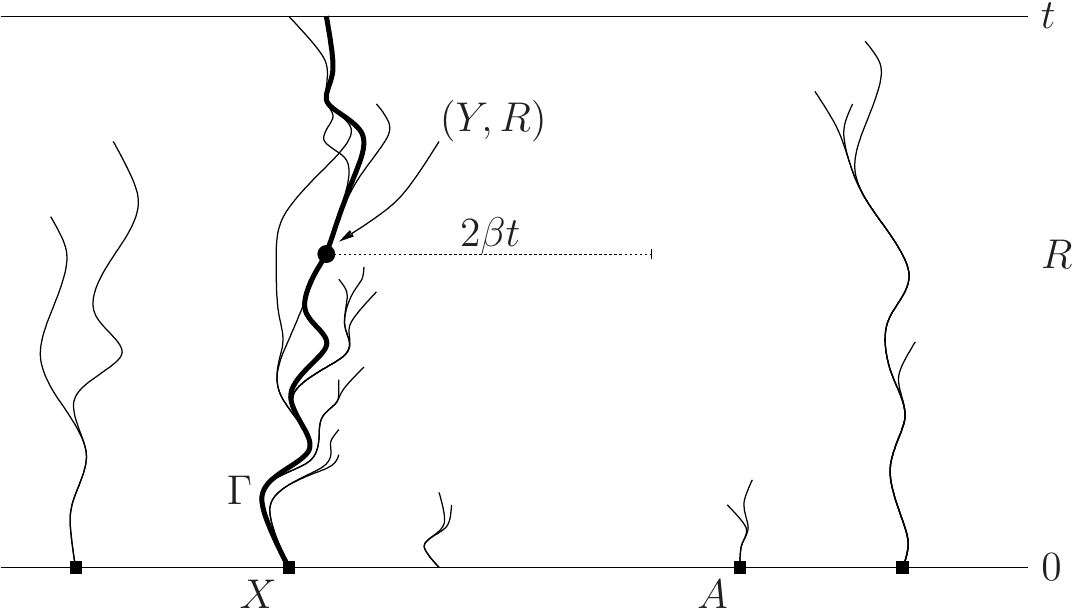}
\caption{%
The set~$A$ is represented by squared points, and $X\in A$ is its rightmost point connected to~$L_t$.
Among all paths from $A\times\{0\}$ to $L_t$, the curve~$\Gamma$ (in bold) is the rightmost one.
The space-time point $(Y,R) \in \Gamma$ is a break point, i.e., no points in $L_0$ are connected to the horizontal segment with width $2\beta t$ lying to the right of $(Y,R)$.
It is also the first point of~$\Gamma$ with such property.
}
\label{figure:definitions}
\end{figure}

Here is a sketch on how the above objects will be used to prove Theorem~\ref{thm:2}.
We want to compare $\L(\zeta_t^A)$ and $\nu$.
The main property of break points is that
\[ \L\left(\, \zeta_t^A \,\big|\, R=s, Y=y, \hat G_s^y \,\right) = \L\left( \, \zeta^{\A}_{t-s} \,\big|\, \mathbf{0} \rightsquigarrow L_{t-s}, G_0 \, \right), \]
which will be explained with Figure~\ref{figindependence}.
Another important property is that, with high probability, $R < \frac{t}{2}$, so one can think of $t-s$ as being a large number.
Using this and BK inequality we can show that
\[ \L\left( \zeta^{\A}_{t-s} \,\big|\, \mathbf{0} \rightsquigarrow L_{t-s}, G_0 \right) \approx \L \left( \zeta^{0}_{t-s} \,\big|\, \mathbf{0} \rightsquigarrow L_{t-s} \right), \]
and the latter converges to~$\nu$ by Proposition~\ref{prop:yaglom}.

In the sequel we state these two properties.

\pagebreak[0]
\begin{proposition}
\label{prop-hasbreakpoints}
If $\beta$ is large enough, then
$\Pb\left( R \leqslant \frac{t}{2} \right) \to1 \text{ as }t\to\infty.$
\end{proposition}

Proposition~\ref{prop-hasbreakpoints} is proved in Section~\ref{sec:findbreakpoints}.
The first lemma below describes a regular conditional distribution\footnotemark\ of $\zeta_t^A$ given $\mathcal{A}$, $Y$ and $R$, on the event that some points are good.

\begin{lemma}
\label{lemma2}
For any $s\in[0,t]$, $y\in\Z$, and $A' \in \sigmainfty_0$,
\[
\mathscr{L}\left( \zeta^{A}_{t} \,\big|\, \mathcal{A}=A', Y=y, R=s, \hat G_{y}^{s} \right)
=
\mathscr{L}\left( \zeta^{A'}_{t-s} \,\big|\, \mathbf{0} \rightsquigarrow L_{t-s}, G_{0} \right).
\]
\end{lemma}

\footnotetext
{
The random elements considered in this paper are a graphical construction~$U$ and sometimes a random initial condition~$\eta_0$, both given by locally finite subsets of an Euclidean space.
Therefore we can assume that $(\Omega,\F)$ is a Polish space, and as a consequence regular conditional probabilities exist.
In particular, conditioning on events such as $\{R=s\}$, $\{\Gamma=\gamma \}$, etc. is well defined.
}

In the sequel we state two lemmas that will fill the remaining technical gaps.
We then prove Theorem~\ref{thm:2} and finally the lemmas.

\begin{lemma}
\label{lemma3}
If $\beta$ is large enough, then, as $t\to\infty$,
\[
\Pb \left( G_X, \hat G_Y^R \right) \to 1
\quad
\text{and}
\quad
\sup_{s\in[0,t]} \sup_{A\in\sigmabar}
\left\|
\mathscr{L}\left( \zeta^{A}_{s} \,\big|\, \mathbf{0} \rightsquigarrow L_{s}, G_{0} \right)
-
\mathscr{L}\left( \zeta^{A}_{s} \,\big|\, \mathbf{0} \rightsquigarrow L_{s} \right)
\right\|_{TV}
\to 0
.
\]
\end{lemma}

\begin{lemma}
\label{lemma4}
As $t\to\infty$,
\[
\sup_{A\in\sigmabar}
\left\|
\mathscr{L}\left( \zeta^{A}_{t} \cap{[-2t,0]} \,\big|\, \mathbf{0} \rightsquigarrow L_{t} \right)
-
\mathscr{L}\left( \zeta^{0}_{t} \cap{[-2t,0]} \,\big|\, \mathbf{0} \rightsquigarrow L_{t} \right)
\right\|_{TV}
\to 0
.
\]
\end{lemma}

Theorem~\ref{thm:2} can now be proved using the preceding results.

\begin{proof}
[Proof of Theorem~\ref{thm:2}]
For a signed measure $\mu$ on $\{0,1\}^{\Z}$, we use $\|\mu\|=\|\mu\|_{TV}$ to denote the total variation norm.
Denote $H_y^s:=\G_y^s \cap \{Y=y,R=s\}$.
Given $A\in \sigmainfty_0$,
\begin{align*}
 \limsup_{t\to\infty} & \left\| \L\left(\zeta_t^A \cap{[-t,0]}\right) - {\nu} \right\| \le
\\
 \le & \limsup_{t\to\infty}
{\Pb\left(\hat G_Y^R,R\le \tfrac{t}{2}\right)}
\left\|\L\left(\zeta_t^A \cap{[-t,0]}\,\big|\,\hat G_Y^R,R\le \tfrac{t}{2}\right) - {\nu} \right\|
+ \left[ 1 -
{\Pb\left(\hat G_Y^R,R\le \tfrac{t}{2}\right)}
\right]
\\
 = & \displaystyle \limsup_{t\to\infty} \Big\| \int_{\Z\times[0,\frac t 2]} \left[ \L\left(\zeta_t^A \cap{[-t,0]} \,\big|\, H_y^s\right) - {\nu} \right] \d \Pb\left(Y=y, R=s \,\big|\, \hat G_Y^R,R\le \tfrac{t}{2}\right) \Big\|\\
 \le & \limsup_{t\to\infty} \sup_{s\le \frac t 2} \sup_{y\in\Z} \left\|\L\left(\zeta_t^A \cap{[-t,0]}\,\big|\,H_y^s\right) - {\nu} \right\| \\
 = & \limsup_{t\to\infty} \sup_{s\le \frac t 2} \sup_{y\in\Z} \left\| \int_{\sigmainftybar} \left[ \L\left(\zeta_t^A \cap{[-t,0]}\,\big|\,H_y^s, \A=A'\right) - {\nu} \right] \d \Pb(\A=A'|H_y^s) \right\|\\
 = & \limsup_{t\to\infty} \sup_{s\le \frac t 2} \left\| \int_{\sigmainftybar} \left[ \mathscr{L}\left( \zeta^{A'}_{t-s} \cap{[-t,0]} \,\big|\, \mathbf{0} \rightsquigarrow L_{t-s}, G_{0}\right) - {\nu} \right] \d \Pb(\A=A'|H_y^s) \right\|\\
 \le & \limsup_{t\to\infty} \sup_{s\le \frac t 2} \sup_{A'\in\sigmabar} \left\| \mathscr{L}\left( \zeta^{A'}_{t-s} \cap{[-t,0]}\,\big|\, \mathbf{0} \rightsquigarrow L_{t-s} \right)
 - {\nu} \right\|\\
 \le & \limsup_{t\to\infty} \,\, \sup_{\frac t 2 \le r \le t} \left\| \mathscr{L}\left( \zeta^{0}_{r} \cap{[-t,0]} \,\big|\, \mathbf{0} \rightsquigarrow L_{r} \right)
 - {\nu} \right\|.
\end{align*}
On the first equality we used Proposition~\ref{prop-hasbreakpoints} and Lemma~\ref{lemma3}.
On the third equality we used Lemma~\ref{lemma2}.
The last two inequalities are due to Lemma~\ref{lemma3} and Lemma~\ref{lemma4}, respectively.
The last $\limsup$ vanishes by Proposition~\ref{prop:yaglom}.
\end{proof}

We proceed to prove the previous lemmas.

\begin{proof}
[Proof of Lemma~\ref{lemma2}]
Consider the regions
\[
E^+_{y,s}= (-\infty,y+\beta t] \times(s,2t], \qquad E^-_{y,s}= \R \times [0,2t] \setminus
E^+_{y,s},
\]
and the random variables $X_{y,s}=\max\{x\in A:(x,0) \rightsquigarrow (y,s)\}$, the first site whose infection reaches $(y,s)$ and $\Gamma_{y,s}:[0,s]\to\Z$ given by
\[
\Gamma_{y,s}(u)=\max\{z:(X_{y,s},0)\rightsquigarrow (z,u)\rightsquigarrow (y,s)\},
\]
the rightmost path from $(X_{y,s},0)$ to $(y,s)$.
Before continuing with the proof, the reader may see Figure~\ref{figindependence} to have a glance of the argument.

\begin{figure}[h!]
\centering\includegraphics[width=135mm]{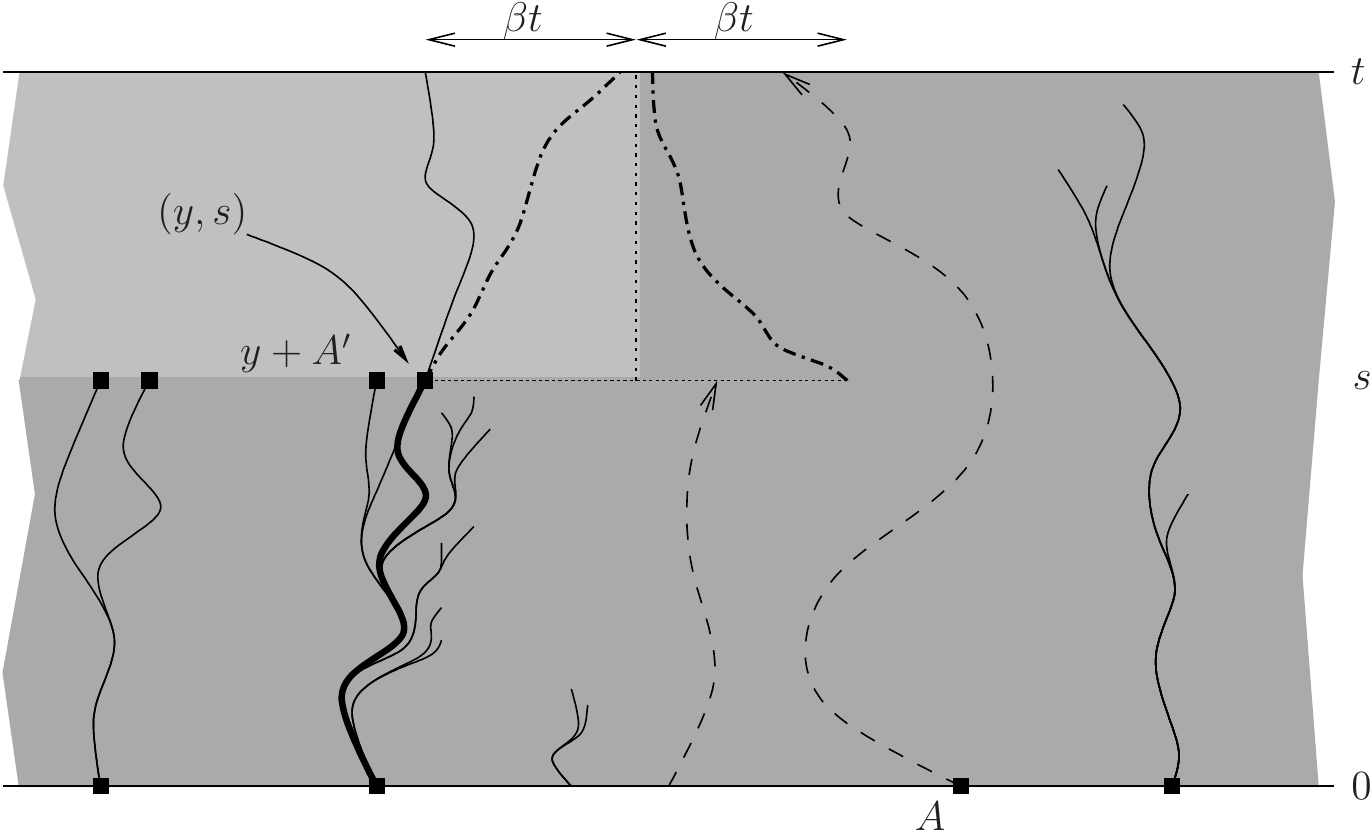}
\caption{The event $\mathcal{A}=A', Y=y, R=s, \hat G_{y}^{s}$.
This event is split in two parts.
The first part depends on the region $E^+_{y,s}$, in light gray, and consists of the occurrence of both $G^s_y$ and $\left\{(y,s) \rightsquigarrow L_t \right\}$.
A consequence of the former is that the rightmost $\lambda$-path starting at $(y,s)$, depicted by a dash-dotted line, does not reach distance $\beta t$ by time $t$, and the latter is represented by a solid arrowed line from $(y,s)$ to $L_t$.
The second part depends on the region $E^-_{y,s}$, in dark gray, and consists of the occurrence of both $H \cap G_{y+2\beta t}^{s}$, as well as an open path from $A \times \{0\}$ to each point in $y+A'$.
The latter is depicted by solid arrowed lines from $A$ at time $0$ to $y+A'$ at time $s$.
The event $H$ breaks down to the following conditions being satisfied.
First, there are open paths from $A$ to $(y,s)$, of which the rightmost one $\Gamma_{y,s}(\cdot)$ is depicted in a thick arrowed line.
Second, there is no open path from $L_0$ to the segment of size $2\beta t$ to the right of $(y,s)$, and moreover $s$ is the first time with this property, i.e., there are open paths from $L_0$ to $\Gamma_{y,s}(u)+z(u)$ for some $0\le z(u)\le 2\beta t$ for all $u<s$.
Third, there are no connections from $A\times \{0\}$ to $L_t$ to the right of $\Gamma_{y,s}$.
Even though the third condition might not depend only on the region $E^-_{y,s}$, it is the case when $G_{y+2\beta t}^{s}$ occurs, since it implies that the leftmost path starting from $(y+2\beta t,s)$, also depicted by a dash-dotted line, does not reach distance $\beta t$ by time $t$.
Finally, $E^+_{y,s}$ and $E^-_{y,s}$ are disjoint, and under the occurrence of $G^s_y$ the configuration $\zeta^{A}_t=\zeta^{A'}_{s,t}$ depends only on $E^+_{y,s}$.
Therefore only the second part ($G^s_y \cap \left\{(y,s) \rightsquigarrow L_t \right\}$) influences its distribution.
}
\label{figindependence}
\end{figure}

For a configuration $\eta \subseteq \Z$ we use the following convention: $\eta \cdot 1 = \eta$ and $\eta \cdot 0=\emptyset$.
We have
\begin{align*}
\mathscr{L}\left( \zeta^{A}_{t} \,\big|\, \mathcal{A}=A', Y=y, R=s, \hat
G_{y}^{s} \right) = \mathscr{L}\left( \zeta^{A' + y}_{s,t} \cdot \I(G_y^s) \,\big|\,
\mathcal{A}=A', (y,s) \rightsquigarrow L_t, H, \hat G_{y}^{s}\right),
\end{align*}
where
\begin{align*}
H = & \{ (A\times\{0\}) \rightsquigarrow (y,s)\}\cap \{L_0 \not \rightsquigarrow (y,y+2\beta t] \times \{s\}\}\\
 & \cap \displaystyle \bigcap_{u \in [0,s)} \{ [X_{y,s},+\infty)\times\{0\} \rightsquigarrow(\Gamma_{y,s}(u), \Gamma_{y,s}(u) + 2\beta t]\times\{u\} \}\\
 & \cap \{\eta^A_s \cap (y+2\beta t,+\infty)\times\{s\} \not \rightsquigarrow L_t\}.
\end{align*}

Observe that $H\cap \{\A=A'\}\cap G_{y+2\beta t}^{s}$ depends on $U\cap E_{y,s}^-$. Since $\zeta^{A'}_{s,t} \cdot \I(G_y^s)$ depends on $U\cap E_{y,s}^+$, we have
\begin{align*}
& \mathscr{L}\left( \zeta^{A}_{t} \,\big|\, \mathcal{A}=A', Y=y, R=s, \hat
G_{y}^{s} \right) = \mathscr{L}\left( \zeta^{A' + y}_{s,t} \cdot \I(G_y^s) \,\big|\,
(y,s) \rightsquigarrow L_t, G_{y}^{s} \right) = \\
& \mathscr{L}\left( \zeta^{A'+y}_{s,t} \,\big|\, (y,s) \rightsquigarrow L_t, G_{y}^{s} \right) = \mathscr{L}\left( \zeta^{A'}_{t-s} \,\big|\, \mathbf{0} \rightsquigarrow L_{t-s}, G_{0} \right),
\end{align*}
by translation invariance.
\end{proof}

\begin{proof}
[Proof of Lemma~\ref{lemma3}]
The two limits hold for similar reasons.
First notice that the probability that $(y,0) \rightsquigarrow L_t$ by a straight vertical path is $e^{-t}$ and that these events are independent over~$y$.

Let $x(k)$ denote the $k$-th point of $A$ from the right.
By independence, $\Pb[X<x(k)] \leqslant (1-e^{-t})^{k} \leqslant e^{-ke^{-t}}$ and writing $x(r)=x(\lfloor r \rfloor)$ we have $\Pb\big(X<x(e^{2t})\big) \to 0$.
Observe that if $(y,0)$ is $(2t,\frac{\beta}{2})$-good then $(y,s)$ is $(t,\beta)$-good for every $s \in [0,t]$.
So we can pick a $\rho>3$ and a $\beta$ according to Lemma~\ref{lemma-good} below to obtain $\Pb\big(\hat G_y^s \text{ for all } s\in[0,t]\big) > 1 - 2e^{-\rho t}$ for large enough~$t$. Hence,
\[\Pb\big(\hat G_y^s \text{ for all } y \in A \cap [x(e^{2t}),0] \text{ and }s\in[0,t]\big) \to 1,\]
and the first limit holds.

Finally,
$
\Pb(G_0^c) \leqslant
e^{-\rho t}
\ll
e^{-t}
\leqslant e^{-s} \leqslant
\Pb(\mathbf{0} \rightsquigarrow L_{s}),
$
and therefore
\begin{equation}
\label{eq:stillgood}
\lim_{t\to\infty} \inf_{s\le t}\Pb(G_0 \,|\, \mathbf{0} \rightsquigarrow L_{s} )= 1,
\end{equation}
proving the second limit.
\end{proof}

\begin{lemma}
\label{lemma-good}
For every $\rho<\infty$, one can choose $\beta<\infty$ such that
\[
\Pb(\mathbf{0} \text{ is good}\,) > 1-e^{-\rho t} \quad \text{for large enough $t$}.
\]
\end{lemma}
\begin{proof}
Given the Poisson process~$U$, the $\lambda$-paths starting at~$\mathbf{0}$ can be constructed by choosing at each jump mark whether or not to follow that arrow.
This way each finite path is associated to a finite binary sequence $a \in \{0,1\}^n$ for some $n\in\N$.

The $\lambda$-path corresponding to a finite sequence~$a$ makes $|a| := \sum_{i=1}^n a_i$ jumps.
Such path is performed in time~$T_a$, whose distribution is that of the sum of~$n$ independent exponential random variables with parameter~$2\lambda$.

Since
\[
\Pb(T_a \le t) = e^{-2\lambda t}\sum_{k=n}^\infty \frac{(2\lambda t)^k}{k!} \le \frac{(2\lambda t)^n}{n!}
\]
for every $a\in\{0,1\}^n$,
choosing $\beta > \max\{12\lambda e,\rho\}$, we have by Stirling's approximation
\begin{multline*}
\Pb(\mathbf{0} \text{ is not good})
\le
\sum_{|a| \ge \beta t} \Pb(T_a \le t)
\le
\sum_{n\ge \beta t} \sum_{a\in\{0,1\}^n} \Pb(T_a \le t)
=
\sum_{n\ge \beta t} 2^n \frac{(2\lambda t)^n}{n!}
\le
\\
\le
\sum_{n\ge \beta t} (4 \lambda t)^n\left(\frac{3}{n}\right)^n
\le \sum_{n\ge \beta t} \left(\frac{12\lambda t}{12\lambda t e}\right)^n \le \frac{1}{1-e^{-1}} e^{-\beta t} \le e^{-\rho t}
\end{multline*}
for all $t$ large enough.
\end{proof}

\begin{proof}
[Proof of Lemma~\ref{lemma4}]
By monotonicity on $A' \in \sigmabar$, it suffices to show that
\(
\Pb( H \,|\, \mathbf{0} \rightsquigarrow L_{t} )
\to 0
\text{ as }
t \to \infty
,
\)
where
$H=\left\{ \zeta^{A'}_{t}(x) \ne \zeta^{0}_{t}(x) \text{ for some } x\in[-2t,0] \right\}$
and $A' = -\N_0$.

Assume that the events $\{ \mathbf{0} \rightsquigarrow L_{t} \}$ and $G_0$ occur (the value of $\beta$ will be fixed later).
Then the rightmost point~$Y$ of $\eta_t$ satisfies $-\beta t < Y < \beta t$.
Therefore, if~$H$ also occurs then $(A' \times \{0\}) \rightsquigarrow (z,t)$ and $\mathbf{0}\not \rightsquigarrow (z,t)$ for some $z\in[-2t-\beta t,\beta t]$, which in turn implies that $({A'}\times\{0\}) \rightsquigarrow (z,t)$ and $\mathbf{0} \rightsquigarrow L_t$ by disjoint paths.
Let $E_n = [-n,n]\times [0,t]$.
Using the BK inequality,
\begin{align*}
\Pb ( H \cap G_0 \,|\, \mathbf{0} \rightsquigarrow L_{t} )
&
\le
\sum_z
\Pb\left( ({A'}\times\{0\}) \rightsquigarrow (z,t) \,\Square\, \mathbf{0} \rightsquigarrow L_t \,\big|\, \mathbf{0} \rightsquigarrow L_{t} \right)
\\ &
=
\sum_z
\lim_n
\frac
{\Pb\left( ({A'}\times\{0\}) \rightsquigarrow (z,t) \,\Square\, \mathbf{0} \rightsquigarrow L_t \text{ in } E_n \big. \right)}
{\Pb\left( \mathbf{0} \rightsquigarrow L_{t} \right)}
\\ &
\le
\sum_z
\lim_n
\Pb\left( \big. (A' \times \{0\}) \rightsquigarrow (z,t) \text{ in } E_n \right)
\frac
{\Pb\left( \mathbf{0} \rightsquigarrow L_t \text{ in } E_n \right)}
{\Pb\left( \mathbf{0} \rightsquigarrow L_t \right)}
\\ &
=
\sum_z
\Pb\left( \big. (A'\times \{0\}) \rightsquigarrow (z,t) \right)
\le
( 2 \beta t + 2 t ) e^{-\alpha t}
.
\end{align*}
Choosing~$\beta$ large enough so that~(\ref{eq:stillgood}) holds, we get the desired limit.
\end{proof}

\subsection{Existence of break points}
\label{sec:findbreakpoints}

In this section we prove Proposition~\ref{prop-hasbreakpoints}.
To this end we show that there must be several time intervals where the path~$\Gamma$ is reasonably smooth, so that a break point is produced on each such piece with positive probability.

\begin{definition}
[Favorable time intervals]
\label{d1}
Let $\gamma$ be a path in the time interval $[0,t]$ and let $\beta>0$.
We say that a time interval $[s-\sqrt{t},s)\subseteq [0,t]$ is \emph{favorable} for path $\gamma$ if for any $u\in [s-\sqrt{t} ,s)$ the number of jumps of $\gamma$ during $[u,s)$ is at most $4 \beta|s-u|$.
\end{definition}

\begin{lemma}
\label{lemmafavorable}
Let $\gamma$ be a path in the time interval $[0,t]$ with at most $\beta t$ jumps.
Then there are at least $\frac{\sqrt{t}}{4}-1$ disjoint favorable intervals for $\gamma$ contained in $[0,\frac{t}{2}]$.
\end{lemma}

To prove this lemma we will seek favorable intervals in a top-down fashion, and use the fact that the existence of a non-favorable interval requires too many jumps.
Notice that, on the event that~$X$ is good, any $\lambda$-path starting from $(X,0)$ makes less than~$\beta t$ jumps, and in particular we can apply this lemma to the path~$\Gamma$.

Define the sets $C_t=\{(x,0):x=1,2,\dots,2\beta t\}$ and
\[
D_t = \left\{ (x,-u) : x=\lfloor 4 \beta u \rfloor, u\in[0,\sqrt{t}] \right\} \cup \left\{ (x,-\sqrt{t}):x \geqslant 4\beta\sqrt{t} \right\},
\]
shown in Figure~\ref{figctdt}.
The following fact is a direct consequence of exponential decay for the subcritical contact process.
It will be proved in the end of this section for convenience.
\begin{figure}[ht!]
\centering
\includegraphics[width=130mm]{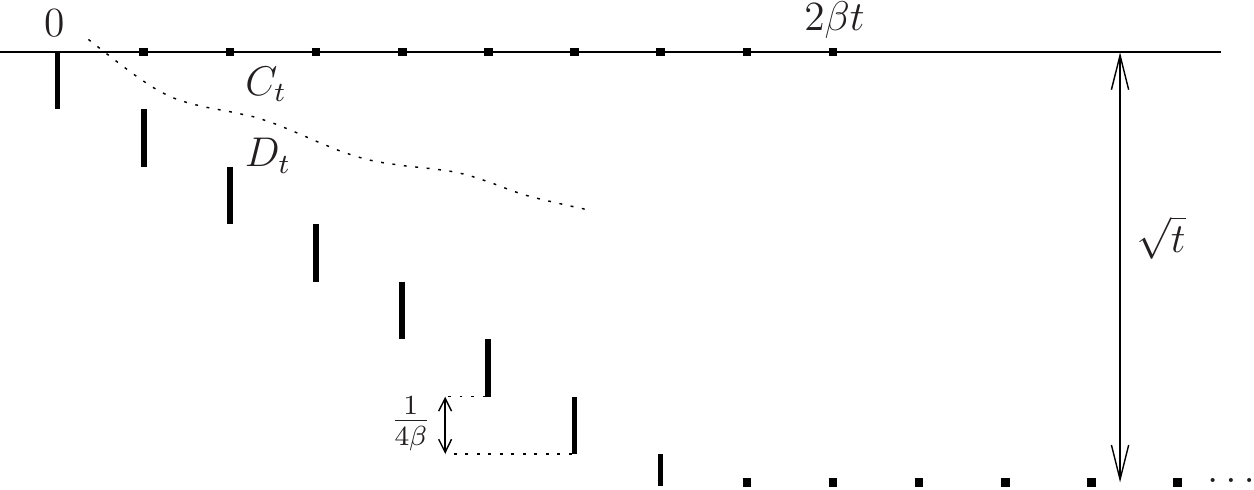}
\caption{Sets $C_t$ and $D_t$}
\label{figctdt}
\end{figure}

\begin{lemma}
\label{l3}
For any $\beta$, let
\(
p_\beta := \sup_{t \geqslant 1} \Pb\left( D_t \rightsquigarrow C_t \right)
.
\)
Then $p_\beta<1$.
\end{lemma}

The key step in proving Proposition~\ref{prop-hasbreakpoints} is to observe that~$X$ is good with high probability, so that one can apply Lemma~\ref{lemmafavorable} combined with the following fact.

\begin{lemma}
\label{l4}
If a path $\gamma$ in the time interval $[0,t]$ has at least $k$ disjoint favorable intervals in $[0,\frac{t}{2}]$, then
\(
\Pb \left( R \leqslant \tfrac{t}{2} \,\big|\, \Gamma=\gamma \right) \geqslant 1-p_\beta^k
\text{ for all } t \geqslant 1.
\)
\end{lemma}

In order to prove the above lemma, we will attach a copy of~$C_t$ and~$D_t$ to disjoint pieces of~$\Gamma$ corresponding to a favorable time interval, and observe that, in order to find a break point, it suffices to have $D_t \not\rightsquigarrow C_t$, see Figure~\ref{figcrossingdt}.
Knowing the path~$\Gamma$ gives negative information about connectivity properties to the right of~$\Gamma$ itself, which by the FKG inequality will increase the probability of the event $D_t \not\rightsquigarrow C_t$.

We are ready to prove Proposition~\ref{prop-hasbreakpoints}.

\begin{proof}
[Proof of Proposition~\ref{prop-hasbreakpoints}]
Let $H$ denote the event that~$\Gamma$ makes less than~$\beta t$ jumps.

By Lemmas~\ref{lemmafavorable} and~\ref{l4},
\[
\Pb\left( R \leqslant \tfrac{t}{2} \right)
\geqslant
\E \left[ \Pb\left( R \leqslant \tfrac{t}{2}, H \,\big|\, \Gamma \right) \right]
\geqslant
\big( 1 - {p_\beta}^{\frac{\sqrt{t}}{4}-1} \big) \cdot \Pb(H)
.
\]
By Lemma~\ref{lemma3}, we can choose~$\beta$ large enough so that $\Pb(H) \geqslant \Pb(G_X) \to 1$ as $t\to\infty$, and the result follows.
\end{proof}

We finish this section with the proof of the lemmas above.

\begin{proof}
[Proof of Lemma~\ref{lemmafavorable}]
We split time interval $[0,\frac{t}{2})$ into a collection of favorable and non-favorable intervals as follows.
Let $t_0=\frac{t}{2}$ and let
\[
v=\inf \big\{u \geq 0: \gamma \mbox{ has more than } 4 \beta u \mbox{ jumps in time interval }
[t_0-u,t_0) \big\}.
\]
If $v>\sqrt{t}$, the interval $[t_0-\sqrt{t},t_0)$ is favorable and we let $t_1=t_0-\sqrt{t}$.
If not, we declare the interval $[t_0-v,t_0)$ non-favorable, and let $t_1=t_0-v$.
We then continue with $t_1$ playing the role of $t_0$, to find $I_2=[t_2,t_1)$ which may be favorable or non-favorable, and so on.
This algorithm is performed until we reach a $t_i<\sqrt{t}$.

Let $L$ denote the sum of the lengths of the non-favorable intervals among $I_1,\dots,I_i$.
Note that a non-favorable interval of length $\ell$ has at least $4\beta\ell$ jumps, so $4 \beta L \le \beta t$ and therefore $L\leq \frac{t}{4}$.
Hence the sum of the lengths of the favorable intervals among $I_1,\dots,I_i$
is at least
$\frac{t}{2}-\sqrt{t}-\frac{t}{4}$ and there must be at least $\frac{\sqrt{t}}{4}-1$ favorable intervals among the $I_j$'s.
\end{proof}

\begin{proof}
[Proof of Lemma~\ref{l4}]
Let $D_\gamma$ be the closed set given by the union of the horizontal and vertical segments of $\gamma$.
Then $(\R \times [0,t]) \setminus D_\gamma$ has two components: $D_\gamma^+$ to the right and $D_\gamma^-$ to the left.

We note that $\left\{ \Gamma = \gamma \right\} = \left\{ \gamma \text{ is open} \right\} \cap H_\gamma^c$, where
\[
H_\gamma =
\left\{ L_0 \rightsquigarrow L_t \text{ in } D_\gamma^+ \right\}
\cup
\left\{ L_0 \rightsquigarrow \gamma \text{ in } D_\gamma^+ \right\}
\cup
\left\{ \gamma \rightsquigarrow L_t \text{ in } D_\gamma^+ \right\}
\cup
\left\{ \gamma \rightsquigarrow \gamma \text{ in } D_\gamma^+ \right\}
.
\]
Here the last event means that there is an open path starting and ending at different points of $\gamma$, whose existence is determined by the configuration~$\omega\cap D_\gamma^+$, see Figure~\ref{figgamma}.
\begin{figure}[ht!]
\centering
\includegraphics[height=40mm]{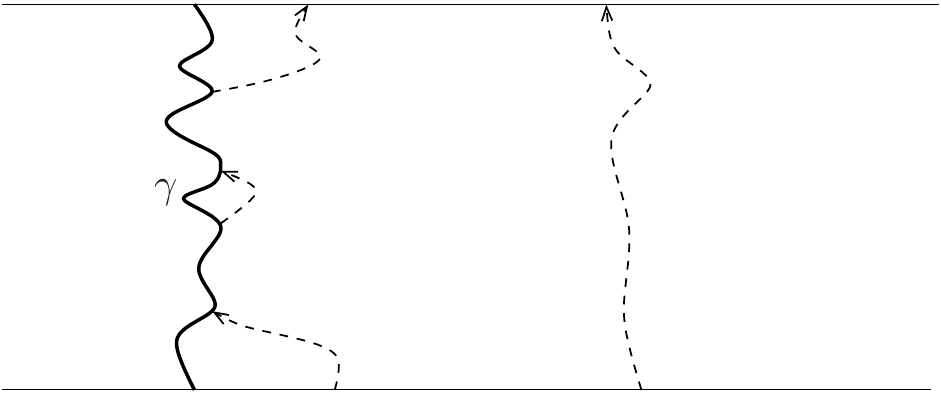}
\caption{The event $H_\gamma$ occurs if a path such as these four is open}
\label{figgamma}
\end{figure}

On the other hand, $\left\{ \Gamma = \gamma\right\} \cap \left\{(y,s) \text{ is a break point} \right\} = \left\{ \Gamma=\gamma \right\} \cap J_{y,s,\gamma}^c$, where
\[
J_{y,s,\gamma} =
\left\{ \gamma \rightsquigarrow (y,s)+C_t \text{ in } D_\gamma^+ \right\}
\cup
\left\{ L_0 \rightsquigarrow (y,s)+C_t \text{ in } D_\gamma^+ \right\}
,
\]
see Figure~\ref{figbreak}.

\begin{figure}[ht!]
\centering
\includegraphics[height=40mm]{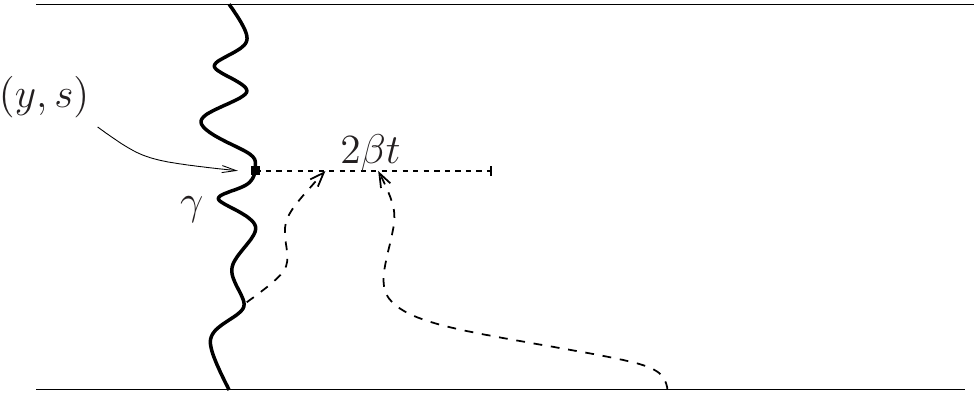}
\caption{The event $J_{y,s,\gamma}$ occurs if a path such as these two is open}
\label{figbreak}
\end{figure}

Now the event $\left\{ \gamma \text{ is open} \right\}$ depends on $U \cap D_\gamma$ and the events $H_\gamma$ and $J_{y,s,\gamma}$ depend on $U \cap D_\gamma^+$.
Since $D_\gamma$ and $D_\gamma^+$ are disjoint,
\begin{align*}
\Pb\left( R \leqslant \tfrac{t}{2} \,\big|\, \Gamma=\gamma \right)
& = \Pb\left( J_{y,s,\Gamma}^c \text{ for some } s\in[0,\tfrac{t}{2}], y=\Gamma(s) \,\big|\, \Gamma=\gamma \right)
\\
& = \Pb\left( J_{y,s,\gamma}^c \text{ for some } s\in[0,\tfrac{t}{2}], y=\gamma(s) \,\big|\, \Gamma=\gamma \right)\\
& = \Pb\left( J_{y,s,\gamma}^c \text{ for some } s\in[0,\tfrac{t}{2}], y=\gamma(s) \,\big|\, \{\gamma \text{ is open }\} \cap H_\gamma^c \right).
\\
& =
\Pb\left( J_{y,s,\gamma}^c \text{ for some } s\in[0,\tfrac{t}{2}], y=\gamma(s) \,\big|\, H_\gamma^c \right).
\end{align*}
Finally, applying the FKG inequality to the last line,
\[
\Pb\left( R \leqslant \tfrac{t}{2} \,\big|\, \Gamma=\gamma \right)
\geqslant
\Pb\left( J_{y,s,\gamma}^c \text{ for some } s\in[0,\tfrac{t}{2}], y=\gamma(s) \right)
.
\]
From now on we drop the subindex $\gamma$ from $J$.
Let $\sqrt{t} \leq t_1<t_2<\cdots<t_k \leqslant \frac{t}{2}$ be such that $t_{j} \geqslant t_{j-1} + \sqrt{t}$ and $[t_j-\sqrt{t},t_j)$ is a favorable interval for $\gamma$.
Write $\boldsymbol{z}_j = (t_j,\gamma(t_j))$.

By definition of favorable interval and of the set $D_t$, we have that $\boldsymbol{z}_j + D_t \subseteq D_\gamma^+ \cup D_\gamma$.
On the other hand, if $J_{\boldsymbol{z}_j}$ occurs then $\boldsymbol{z}_j + D_t \rightsquigarrow \boldsymbol{z}_j + C_t$, see Figure~\ref{figcrossingdt}.

\begin{figure}[ht!]
\centering
\includegraphics[width=130mm]{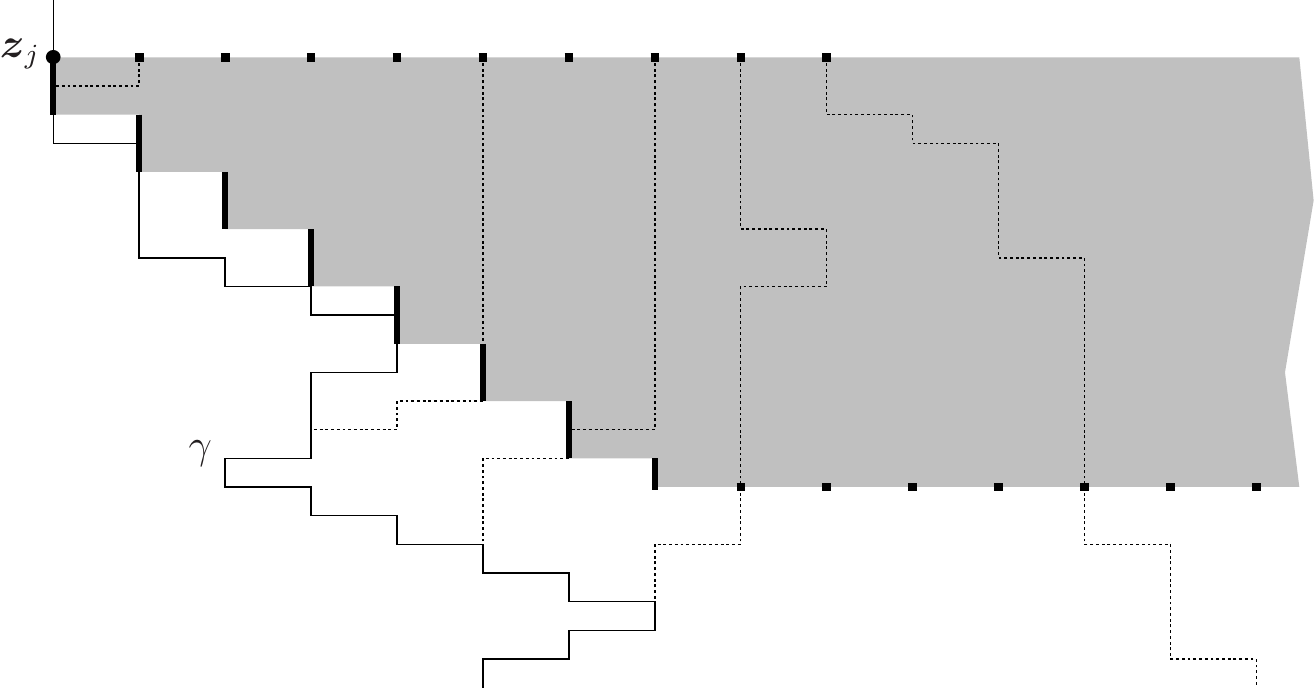}
\caption{The event $J_{\boldsymbol{z}_j}$ implies that $\boldsymbol{z}_j + D_t \rightsquigarrow \boldsymbol{z}_j + C_t$}
\label{figcrossingdt}
\end{figure}

Since these events depend on $U \cap \big(\R \times (t_j-\sqrt{t},t_j]\big)$, which are disjoint as $j$ goes from 1 to $k$, we have that
\[
\Pb\left( R \leqslant \tfrac{t}{2} \,\big|\, \Gamma=\gamma \right)
\geqslant
1-\Pb\left( \boldsymbol{z}_j + D_t \rightsquigarrow \boldsymbol{z}_j + C_t \text{ for all } j \right)
=
1 - \Pb\left( D_t \rightsquigarrow C_t \right)^k
\geqslant
1 - p_\beta^k
\]
by Lemma~\ref{l3}, which finishes the proof.
\end{proof}

\begin{proof}
[Proof of Lemma~\ref{l3}]
This lemma is a simple consequence of Lemma~\ref{lemmaexp} below.
We give a full proof for convenience.
If $D_t \rightsquigarrow C_t$ then either $(x,-u) \rightsquigarrow C_t$ for some $x\in\Z$ and $u=\sqrt{t}$, or $(x,-u) \rightsquigarrow L_0^+$ for some $x=0,1,2,\dots$ and $u \geqslant x/4\beta$, where $L_0^+ = \{1,2,3,\dots\}\times\{0\}$.
Using \eqref{bound.above.nonabsorption} and summing over $y \in C_t$, the
probability of the first event is bounded by $2 \beta t e^{-\alpha \sqrt{t}}$.
For the second event, using FKG inequality and Lemma~\ref{lemmaexp} below we get
\begin{multline*}
q_\beta
:=
\Pb\left( (x,-u) \not\rightsquigarrow L_0^+ \text{ for any } x=0,1,2,\dots \text{ and } u \geqslant x/4\beta \right)
\geqslant
\\
\geqslant
\prod_x
\Pb\left( (x,-u) \not\rightsquigarrow L_0^+ \text{ for any } u \geqslant x/4\beta\right)
>
0
\end{multline*}
does not depend on~$t$.
By the FKG inequality, $p_\beta \leqslant \sup_{t \geqslant 1} 2\beta t
e^{-\alpha \sqrt{t}} (1-q_\beta) < 1$.
\end{proof}

\begin{lemma}
\label{lemmaexp}
For large enough $t$,
\(
\Pb \left( \big. L_0 \rightsquigarrow (0,s) \text{ for some } s \geqslant t \right) \leqslant e^{-\alpha t/2}
.
\)
\end{lemma}
\begin{proof}
On the one hand the existence of a QSD ${\nu}$, Proposition~\ref{prop:yaglom}, implies
\begin{equation}
\label{bound.above.nonabsorption}
\Pb (L_{0} \rightsquigarrow (0,t))
=
\Pb(\mathbf{0} \rightsquigarrow L_t)
=
\Pb(\tau^0>t)
\le
\Pb(\tau^{\nu}>t)
=
e^{-\alpha t}.
\end{equation}
On the other hand
\begin{align*}
\Pb\left( \big. L_0 \rightsquigarrow (0,t+1) \right)
&
\geqslant
\Pb\left( \big. L_0 \rightsquigarrow (0,s) \text{ for some } s \in [t,t+1] \right)
\Pb\left( \big. U^x \cap [t,t+1] = \emptyset \right)
\end{align*}
and $\Pb\left( \big. U^x \cap [t,t+1] = \emptyset \right) = e^{-1}$, whence
\begin{align*}
\Pb \left( \big. L_0 \rightsquigarrow (0,s) \text{ for some } s \geqslant t \right)
&
\leqslant
\sum_n
\Pb \left( \big. L_0 \rightsquigarrow (0,s) \text{ for some } s \in [t+n,t+n+1] \right)
\\
&
\leqslant
e \cdot
\sum_n
\Pb \left( \big. L_0 \rightsquigarrow (0,t+n+1) \right)
\leqslant
e^{-\alpha t/2}
\end{align*}
for $t$ large enough.
\end{proof}

\section{Yaglom limit for the set infected by a single site}
\label{sec:yaglom}

In this section we prove Proposition~\ref{prop:yaglom}, building upon Chapter~3 of the second author's PhD thesis~\cite{ezanno12}.

We start by recalling some properties of jump processes on countable spaces which are almost-surely absorbed but positive recurrent when conditioned on non-absorption, known as \emph{$R$-positive} or \emph{$\alpha$-positive} processes.
In the sequel we define the finite contact process modulo translations, extending to~$\Z^d$ the concept of ``seen from the edge''. We then discretize time appropriately to obtain some moment control using exponential decay, showing that it satisfies some probabilistic criteria for $R$-positiveness which ultimately implies the desired result.

\subsection{Positive recurrence of conditioned processes}

Let~$\Lambda$ be a countable set and consider a Markov jump process $(\zeta_t)_{t \ge 0}$ on $\Lambda \cup \{\emptyset\}$ such that~$\Lambda$ is an irreducible class and $\emptyset$ is an absorbing state which is reached almost-surely.
The sub-Markovian transition kernel restricted to~$\Lambda$ is written as $P_t(A,A')=\Pb(\zeta^A_t=A')$, a matrix doubly-indexed by~$\Lambda$ and continuously parametrized by~$t$.

A measure~$\mu$ on~$\Lambda$ is seen as a row vector, and a real function $f$ as a column vector, so that
\(
 \mu P_t f = 	\E f  (\zeta_t^\mu  ).
\)
With this notation, ${\mu}$ is quasi-stationary if and only if ${\mu} P_t = e^{-\alpha(\mu) t} {\mu}$.

By~\cite[Theorem~1]{kingman-63b} there exists $\alpha >0$ with the property that $t^{-1} \log P_t(A,A') \to -\alpha$ as $t\to\infty$ for every $A,A' \in \Lambda$. The semi-group $(P_t)$ is said to be \emph{$\alpha$-positive} if
\[
 \limsup_{t \to \infty} e^{\alpha t} P_t(A,A) >0.
\]
In this case, by~\cite[Theorem 4]{kingman-63b} there exist a measure ${\nu}$ and a positive
function $h$, both unique modulo a multiplicative constant, such that
\[
\nu P_t = e^{-\alpha t} {\nu},
\qquad
P_t h = e^{-\alpha t} h.
\]
Moreover, $\nu h < \infty$.
If in addition $\nu$ is summable, then it can be normalized to become a probability measure on~$\Lambda$, and the Yaglom limit follows from the result below.
\begin{theorem}
\label{thm:verejones}
If an irreducible sub-Markovian standard semi-group $(P_t)_{t\ge 0}$ on a countable space $\Lambda$ is $\alpha$-positive with summable normalized left-eigenvector~$\nu$,
then
\[
\lim_{t\to\infty}\frac{P_t(A,A')}{P_t(A,\Lambda)} = \nu(A'), \quad  \forall A,A' \in \Lambda.
\]
\end{theorem}

\begin{proof}
We reprove this classical result~\cite{seneta-vere-jones-66,vere-jones-69} for the reader's convenience.

Let $\boldsymbol{1}$ denote the unit column vector, and choose $\nu$ and $h$ so that
\[
P_t h = e^{-\alpha t} h,
\quad
\nu P_t = e^{-\alpha t} \nu,
\quad
\nu \boldsymbol{1} = 1,
\quad
\nu h = 1.
\]
Let $H$ denote the diagonal matrix corresponding to $h$.
The \emph{$h$-transform} of $P_t$ is
\[
Q_t=e^{\alpha t} H^{-1} P_t H.
\]
Since $\nu H Q_t = \nu H$, $Q_t \boldsymbol{1} = \boldsymbol{1}$ and $(Q_t)_t$ is a multiplicative semi-group, it defines a Markov process on~$\Lambda$ with invariant measure $\nu H$.

It follows from the $\alpha$-positiveness of $(P_t)_t$ that $Q_t \not\to 0$, thus it is positive recurrent and hence
$Q_t \to \boldsymbol{1} \nu H$
as $t\to\infty$.
Therefore,
\(
e^{\alpha t} P_t \to h \nu \text{ as } t \to \infty.
\)
Summing over the second coordinate we have
\(
e^{\alpha t} P_t \boldsymbol{1} \to h
.
\)
That is,
\[
e^{\alpha t} P_t(A,A') \to h(A) \nu(A')
\quad
\text{and}
\quad
e^{\alpha t} P_t(A,\Lambda) \to h(A)
.
\]
Therefore we get
\[
\frac{P_t(A,A')}{P_t(A,\Lambda)}
\to
\frac{h(A){\nu}(A')}{h(A)} = {\nu}(A').
\]

It remains to justify that summation over the second coordinate preserves the
limit. Since $e^{\alpha t}\nu P_t = \nu$ we get for every $t\ge 0$ and $A'\in\Lambda$
\[
e^{\alpha t}P_t(A,A') =  \frac{e^{\alpha t} \nu(A)  P_t(A,A')}{\nu(A)}\le \frac{\nu(A')}{\nu(A)},
\]
which is summable over $A'$.
The limit thus follows by dominated convergence.
\end{proof}

\subsection{Finite contact process modulo translations}

For the contact process on~$\Z^d$ in arbitrary dimension $d \ge 1$, the concept of ``{seen from the edge}'' is generalized by considering the process modulo translations.
We say that two configurations~$\eta$ and~$\eta'$ in the space $\{ A \subseteq \Z^d : A \text{ is non-empty and finite} \}$ are \emph{equivalent} if $\eta=\eta'+y$ for some $y\in\Z^d$.
Let~$\Lambda$ denote the quotient space resulting from this equivalence relation.
We will denote by~$\zeta$ the equivalence class of a configuration~$\eta$, or indistinguishably any representant of such class when there is no confusion.

Since the evolution rules of the contact process are translation-invariant, the process $(\zeta_t)_{t\ge 0}$ obtained by projecting $(\eta_t)_{t\ge 0}$ onto~$\Lambda \cup \{\emptyset\}$ is a homogeneous Markov process with values on $\Lambda \cup \{\emptyset\}$.
Moreover, the subset~$\Lambda$ is an irreducible class, and the absorbing state~$\emptyset$ is almost-surely reached if $\lambda<\lambda_c$.
We call $(\zeta_t)_{t\ge 0}$ the \emph{contact process modulo translations}.

For $d=1$ this is the same as taking $\Lambda = \{ A \subseteq -\N_0 : A \text{ is finite and } 0\in A \}$.
Therefore, Proposition~\ref{prop:yaglom} is the specialization to $d=1$ of the next result.

\begin{proposition}
\label{prop:zdyaglom}
Let $(\zeta_t)_{t\geqslant 0}$ denote the contact process modulo translations on~$\Z^d$ with subcritical infection parameter~$\lambda$.
This process has a unique minimal quasi-stationary distribution~${\nu}$.
Moreover the Yaglom limit
\(
\L ( \zeta_t^A \,|\, \tau^A>t ) \to {\nu} \text{ as } t\to\infty
\)
holds for any finite non-empty initial configuration~$A$.
\end{proposition}

Let $(\xi_n)_n$ denote an irreducible, aperiodic, discrete-time Markov chain on the state-space~$\Lambda \cup \{\emptyset\}$, with transition matrix~$p(\cdot,\cdot)$ such that the absorbing time~$\tau^A=\inf\{n:\xi_n^A=\emptyset\}$ is a.s.\ finite.
As for continuous-time chains discussed above, there is~$R$ such that $p^n(A,A) = R^{-n+o(n)}$, and we say that~$p$ is \emph{$R$-positive} if $\limsup R^{n}p^n(A,A)>0$.
The proof of Proposition~\ref{prop:zdyaglom} will be based on the following criteria for $R$-positiveness.

\begin{theorem}
[{\cite[Theorem~1]{ferrari-kesten-martinez-96}}]
\label{thm:h1h2h3}
Suppose that there exist a subset $\Lambda'\subseteq \Lambda$, a configuration $A'\in\Lambda'$, some $\rho<R^{-1}$, and positive constants~$M$ and~$\varepsilon$ such that
\\
\textnormal{(H1)}
For all $A\in\Lambda'$ and $n\ge 0$, $\Pb(\tau^{A}>n;\, \xi^{A}_1,
\dots,\xi^{A}_{n} \notin \Lambda') \le M \rho^{n}$;
\\
\textnormal{(H2)}
For all $A\in\Lambda'$ and $n\ge 0$, $\Pb(\tau^{A}>n) \le 
M \, \Pb(\tau^{A'}>n)$;
\\
\textnormal{(H3)}
For all $A \in \Lambda'$, $\Pb(\xi^A_n=A' \text{ for some } n\le M) \ge
\varepsilon$.
\\
Then $p$ is $R$-positive and its left eigenvector~$\nu$ is summable.
\end{theorem}

The following proposition provides a set of configurations~$\Lambda'$ and an appropriate time discretization that satisfy the above criteria.
It is analogous to Theorem~2 in~\cite{ferrari-kesten-martinez-96}, but since the range of interaction of the contact process is infinite for any positive period of time, we cannot apply the latter directly.
We give a simpler proof instead.

\begin{proposition}
\label{prop:recurrence}
For a subcritical contact process~$(\eta_t)_{t \ge 0}$ on~$\Z^d$, there exists a time $\psi>0$ with the following property.
For every $\rho>0$, one can find constants~$K$ and~$M$ such that
\begin{equation}
\label{recurrence}
\Pb\left( |\eta^A_{\psi}| >K,\dots, |\eta^A_{n \psi}|>K \right) \le M \rho^n
\end{equation}
for all $A \subseteq \Z^d $ with $1 \le |A|\le K$ and $n\ge 1$.
\end{proposition}

\begin{proof}
A consequence of the exponential
decay for the set of points infected from the origin
is that
$\int_0^{\infty}[\E|\eta_t^0|^q] \dd t< \infty$ for every $q > 0$~\cite[(1.13)-(1.14)]{bezuidenhout-grimmett-91}.
Therefore we can choose a time $\psi>0$ such that $\E|\eta_\psi^0|<1$, and such that $\E|\eta_\psi^0|^q<\infty$ for every~$q$.

We claim that $|\eta^A_t|$ is stochastically bounded by the sum of $|A|$ independent copies of~$|\eta^0_t|$, which we denote by $\eta^{(i)}_t$, $i=1,\dots,|A|$.
The proof of this fact is standard and will be omitted.
By the Law of Large Numbers in $L^q$,
\[
\frac{|\eta^A_\psi|}{|A|}
\ \stackrel{\text{st}}{\leq} \
\frac{1}{|A|} \sum_{i=1}^{|A|} |\eta^{(i)}_\psi|
\stackrel{L^q}{\longrightarrow}\ \E|\eta_\psi^0| \text{ as } |A|\to \infty.
\]

Let $\rho>0$. Choosing $q$ large so that $(\E|\eta_\psi^0|)^q < \rho$, there exist $C$ and $K$ such that
\begin{equation}
\label{eq:moment}
\E\frac{|\eta^A_\psi|^q}{|A|^q} \le C \text{ for all } A\in\Lambda
\quad
\text{ and moreover }
\quad
\E\frac{|\eta^A_\psi|^q}{|A|^q} \le \rho \text{ whenever } |A|>K.
\end{equation}

Writing $\xi_n$ for $\eta_{n \psi}$, for any $A\in\Lambda$ with $|A|\le K$,
\begin{align*}
\Pb\big(|\xi^A_{1}|>K & ,\dots,  |\xi^A_{n}|>K\big)
\le
\frac{1}{K^q}\E\left[|\xi^A_{n}|^q ; |\xi^A_{1}|>K,\dots,
|\xi^A_{n}|>K\right]
\\
&
=
\frac{|A|^q}{K^q}
\E\left[ \frac{|\xi^A_1|^q}{|\xi^A_0|^q}\cdots \frac{|\xi^A_{n-1}|^q}{|\xi^A_{n-2}|^q} \frac{|\xi^A_n|^q}{|\xi^A_{n-1}|^q} ; |\xi^A_{1}|>K,\dots, |\xi^A_{n}|>K \right ]
\\
&
=
\frac{|A|^q}{K^q}
\E \left\{ \E\left[ \frac{|\xi^A_1|^q}{|\xi^A_0|^q}\cdots
\frac{|\xi^A_n|^q}{|\xi^A_{n-1}|^q} ; |\xi^A_{1}|>K,\dots,|\xi^A_{n}|>K \ \Bigg|\  \xi_1,\dots,\xi_{n-1} \right ] \right\}
\\
&
\le
\E\left[ \frac{|\xi^A_1|^q}{|\xi^A_0|^q}\cdots \frac{|\xi^A_{n-1}|^q}{|\xi^A_{n-2}|^q} ; |\xi^A_{1}|>K,\dots, |\xi^A_{n-1}|>K \right ]
\cdot
{ \sup_{|A'|>K} \E \big( \textstyle \frac{|\xi^{A}_n|^q}{|\xi^{A}_{n-1}|^q} \big| \xi^{A}_{n-1}=A' \big) }
\\
&
\le
\rho \cdot \E\left[ \frac{|\xi^A_1|^q}{|\xi^A_0|^q}\cdots
\frac{|\xi^A_{n-1}|^q}{|\xi^A_{n-2}|^q} ; |\xi^A_{1}|>K,\dots,
|\xi^A_{n-1}|>K \right ]
\\
& \le
\cdots
\le
\rho^{n-1}
\cdot
\E \left[ \frac{|\xi^A_1|^q}{|\xi^A_0|^q} ; |\xi^A_{1}|>K \right]
\le
\rho^{n-1}
\cdot
\sup_{A'} \E \frac{|\xi^{A'}_1|^q}{|A'|^q}
\le C \rho^{n-1}.
\end{align*}
We have used~\eqref{eq:moment} $n$ times here.
Writing $M=\frac{C}{\rho}$ the result follows.
\end{proof}

Finally we prove Proposition~\ref{prop:zdyaglom} using the previous results.

\begin{proof}
[Proof of Proposition~\ref{prop:zdyaglom}]
Let~$\psi$ be given by Proposition~\ref{prop:recurrence}.
We now consider the discrete-time chain given by $\xi_n=\zeta_{n \psi}$, $n=0,1,2,\dots$, which has decay rate $R = e^{\alpha \psi}$.
Choose $\rho < R^{-1}$.
By Proposition~\ref{prop:recurrence}, there are~$K$ and~$M$ such that~(\ref{recurrence}) holds, which implies~(H1) with $\Lambda'=\{A \in \Lambda : |A|\le K\}$.

Take $A'=\{0\}$ and observe that $\Pb(\tau^A>n)\le |A| \cdot \Pb(\tau^0>n)$, which implies~(H2).

Finally,~(H3) follows from
\[
\Pb(\zeta^A_\psi=\{0\})\ge e^{-\psi}e^{-2 d \lambda |A|\psi}(1-e^{-\psi})^{|A|-1} \ge \varepsilon,
\quad \text{for every } A \in \Lambda',
\]
where $\varepsilon = [e^{\psi + 2d\lambda\psi}(1-e^{-\psi})]^K > 0$.

By Theorem~\ref{thm:h1h2h3}, the matrix $P_\psi$ is $R$-positive with summable left-eigenvector~$\nu$.
Therefore the semi-group $(P_t)_{t\ge 0}$ is $\alpha$-positive with the same left-eigenvector.
By Theorem~\ref{thm:verejones} $\frac{P_t(A,A')}{P_t(A,\Lambda)} \to \nu(A'), \ \forall A,A' \in \Lambda$, proving Proposition~\ref{prop:zdyaglom}.
\end{proof}

\section*{Acknowledgments}

We thank Santiago Saglietti for helpful suggestions.


\providecommand{\bysame}{\leavevmode\hbox to3em{\hrulefill}\thinspace}
\providecommand{\MR}{\relax\ifhmode\unskip\space\fi MR }
\providecommand{\MRhref}[2]{%
  \href{http://www.ams.org/mathscinet-getitem?mr=#1}{#2}
}
\providecommand{\href}[2]{#2}


\end{document}